\documentclass[11pt]{amsart}
\usepackage{hyperref}
\usepackage{graphicx}
\usepackage{xcolor}
\usepackage{enumitem}

\graphicspath{{graphics/}}

\newcommand{\PP}{\mathbb{P}}
\newcommand{\CC}{\mathbb{C}}
\newcommand{\RR}{\mathbb{R}}

\newcommand{\FF}{\mathbb{F}}

\let\Re\relax
\let\Im\relax
\DeclareMathOperator{\Re}{Re}
\DeclareMathOperator{\Im}{Im}
\DeclareMathOperator{\argmin}{argmin}

\newtheorem{theorem}{Theorem}
\newtheorem{lemma}{Lemma}

\newtheorem{corollary}[theorem]{Corollary}
\newtheorem{conjecture}{Conjecture}

\theoremstyle{definition}

\theoremstyle{definition}
\newtheorem{definition}{Definition}

\theoremstyle{remark}
\newtheorem{remark}{Remark}

\title[A Sylvester-Gallai result for Concurrent Lines in \texorpdfstring{$\CC^2$}{the Complex Plane}]{A Sylvester-Gallai Result for Concurrent Lines in the Complex Plane}
\thanks{This research was done as part of the 2019 CUNY Combinatorics REU, supported by NSF awards DMS-1802059 and DMS-1851420.}
\author{Alex Cohen}
\date{June 2020}
\address{Yale University, 10 Hillhouse Ave, New Haven, CT 06511, USA}
\email{alex.cohen@yale.edu}
\begin{document}
\begin{abstract}
We show that if a finite non-collinear set of points in $\mathbb{C}^2$ lies on a family of $m$ concurrent lines, and if one of those lines contains more than $m-2$ points, there exists a line passing through exactly two points of the set. The bound $m-2$ in our result is optimal. Our main theorem resolves a conjecture of Frank de Zeeuw, and generalizes a result of Kelly and Nwankpa.
\end{abstract}

\maketitle

\section{Introduction}
The Sylvester-Gallai theorem says that any finite set of points in the real plane, not all lying on one line, determines a line passing through exactly two points of the set---these are called \textit{ordinary lines}.
The Sylvester-Gallai theorem has been intensely studied over the real numbers, and many variants have been introduced; see \cite{BMP,GT} for an overview. 

This theorem fails over the complex numbers: there are finite sets of points in $\CC^2$ having no ordinary lines. These sets are known as \textit{complex Sylvester-Gallai configurations}.
One such example is the \textit{Hesse configuration}, which consists of the nine inflection points of an elliptic curve. 
This arrangement realizes the affine plane $\FF_2^3$ in $\CC^2$ and is unique up to a projective automorphism. 
The \textit{Fermat configurations} are an infinite family of examples on $3n$ points generalizing the Hesse configuration; these arise as the inflection points of a Fermat curve, and the points of these configurations always lie on three lines.
Beyond the Fermat configurations, there are two known exceptional configurations in $\CC^2$: the Klein configuration with 21 points and the Wiman configuration with 45 points. 
For more details on these configurations, see \cite{Berger,H,Pokora}.

% These are all of the known complex Sylvester-Gallai configurations, and experts believe that aside from the Fermat configurations, there are only finitely many exceptional configurations. 
% Whereas much is known about Sylvester-Gallai theory in real space, very little is known about Sylvester-Gallai theory in complex space. 
% In particular, even though we expect almost all complex Sylvester-Gallai configurations to belong to one very specific family of examples, there are few results restricting what point sets in $\CC^2$ can be Sylvester-Gallai configurations.

Whereas much is known about Sylvester-Gallai theory in the real plane, very little is known about Sylvester-Gallai theory in the complex plane. One major result on complex line configurations is Hirzebruch's inequality \cite{H}, which shows that any finite set of points in $\CC^2$ not all lying on one line must determine a line passing through two or three points. The proof of this theorem uses a deep inequality from algebraic geometry. Langer \cite{Langer} introduced a stronger version of this inequality along the same lines and using related techniques. Kelly applied Hirzebruch's inequality to prove that there are no complex Sylvester-Gallai configurations lying properly in $\CC^3$ \cite{K}---but beyond Kelly's theorem, Hirzebruch's inequality has not led to further progress in the study of ordinary lines in complex space. 

In another strand of work, Motzkin \cite{Motzkin} initiated a study of Sylvester-Gallai configurations with few points lying in various affine planes, which Kelly \& Nwankpa \cite{KN} extended in order to classify Sylvester-Gallai configurations up to 14 points in planes over various fields. 
In their analysis, Kelly \& Nwankpa proved the following theorem.
\begin{theorem}[Kelly \& Nwankpa, 1973]\label{no_three_concurrent}
There are no Sylvester-Gallai configurations in $\CC^2$ lying on three concurrent lines. 
\end{theorem}
In fact, Kelly \& Nwankpa classified \textit{all} Sylvester-Gallai configurations, over any field, lying on precisely three lines---but we will not need their more general result here. In this paper we extend Kelly \& Nwankpa's theorem to deal with many concurrent lines: we show that Sylvester-Gallai configurations lying on a family of concurrent lines cannot have many points. Our main theorem is the following.

\begin{theorem}\label{concurrent_bound}
If a non-collinear set $S \subset \CC^2$ lies on a family of $m$ concurrent lines, and if one of those lines contains more than $m-2$ points of $S$ (not including the point of concurrency), then the set admits an ordinary line.
\end{theorem}

\begin{figure}
\includegraphics[width=200px]{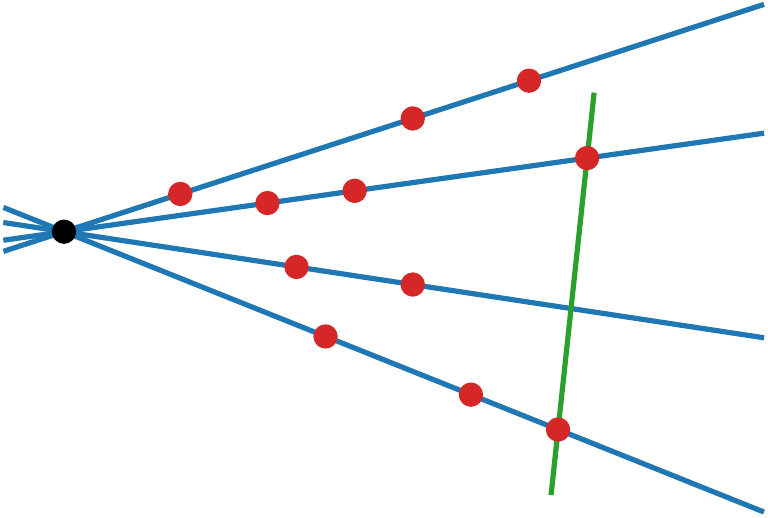}
\caption{A family of points lying on four concurrent lines; the green line is an ordinary line. Theorem \ref{concurrent_bound} says that if one of these lines passes through many points, the set must admit an ordinary line.}
\label{fig:concurrent_lines}
\end{figure}

See Figure \ref{fig:concurrent_lines} for the situation in which Theorem \ref{concurrent_bound} applies. Note that the figures in this paper generally depict geometry in $\RR^2$, but should be interpreted as representing situations in $\CC^2$---they are useful by way of analogy, but our statements and arguments are not particularly interesting when applied to the real plane. 

The $m-2$ bound in this theorem is optimal. Indeed, the Fermat configuration on $3n$ points can be embedded onto $n + 2$ concurrent lines, by choosing the point of concurrency to be one of the intersection points between the three lines the configuration lies on. Two of those $n+2$ lines will pass through $n$ points of the configuration, and the rest pass through one point. Because this is a Sylvester-Gallai configuration, Theorem \ref{concurrent_bound} says that the maximum number of points on one line is $(n+2) - 2 = n$, and that maximum is achieved. 

The proof of Theorem \ref{concurrent_bound} involves ordering complex numbers by their real part, which seems to be a new approach to complex Sylvester-Gallai theory. This approach was motivated by Sylvester-Gallai theory in the real plane: all proofs of the Sylvester-Gallai theorem rely in some way on the ordering of the real numbers, and it turns out that although there is no field-ordering of the complex numbers, an ordering which respects addition suffices to obtain our result. Interestingly, a key step of the proof is an application of an inequality from multivariable calculus applied to a piecewise linear, real-valued function on the complex line---despite several attempts, we were not able to find a purely combinatorial alternative, and it seems that this use of analysis is essential to our proof. 

Theorem \ref{concurrent_bound} resolves Conjecture 4.6 from \cite{Z}.
\begin{corollary}\label{corollary}
The only Sylvester-Gallai configuration that can be embedded on four concurrent lines is the Hesse configuration on nine points.
\end{corollary}
\begin{proof}
By Theorem \ref{concurrent_bound}, if a Sylvester-Gallai configuration lies on four concurrent lines, each line has at most two points aside from the point of concurrency, so the entire configuration has at most nine points. Kelly \& Nwankpa's results \cite{KN} imply that the only possibility is the Hesse configuration.  
\end{proof}

Unpublished computational work by the author suggests that Corollary \ref{corollary} can be extended to five concurrent lines: the only Sylvester-Gallai configurations on up to five concurrent lines are the Fermat configurations on nine and twelve points. For simplicity, we leave out further discussion of this extension and a rigorous proof.

\section{Proof of the theorem}\label{proof_of_theorem_section}

\subsection{Simplest case: new proof of Kelly \& Nwankpa's theorem}\label{section:simplified_kn}
Before we prove Theorem \ref{concurrent_bound}, we give a short proof of Kelly \& Nwankpa's Theorem \ref{no_three_concurrent}: there are no Sylvester-Gallai configurations lying on three concurrent lines. In what follows, we use homogenous coordinates in $\PP\CC^2$, and we use the notation $(x,y) = [x : y : 1] \in \PP\CC^2$ for the embedding of $\CC^2$ in $\PP\CC^2$ for which the line at infinity is $\{[x :y : 0]\ |\ x,y \in \CC, (x,y) \neq (0, 0)\}$. 
\begin{proof}
Let $S$ be a finite set of points in $\PP\CC^2$ lying on three concurrent lines $\ell_1,\ell_2,\ell_3$. Suppose $S$ determines no ordinary lines: we will prove $S$ has at most four points, and then obtain a contradiction. Assume $\ell_3$ has the most points. After applying a projective automorphism, we may assume the common point is $[1 : 0 : 0]$, and $\ell_3$ is the line at infinity. Projective automorphisms preserve all incidence statements, so we lose no generality in performing this transformation. Now $\ell_1$ and $\ell_2$ are lines of the form $y = y_1$, $y = y_2$, and we have
\begin{figure}
\includegraphics[width=300px]{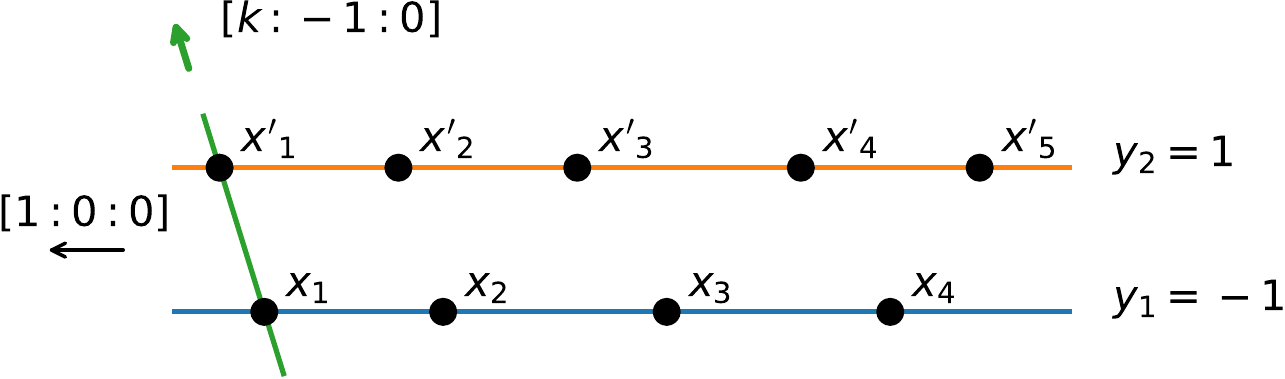}
\caption{The finite points arising in \S\ref{section:simplified_kn}. The green line is the line of slope $k$ with minimal $x$-intercept.}
\label{fig:kn_fig}
\end{figure}
\begin{equation*}
    \ell_1 \cap S = \{(x_1,y_1),\ldots,(x_A,y_1)\},\quad \ell_2 \cap S = \{(x_1',y_2),\ldots,(x_B',y_2)\}
\end{equation*}
where $x_j$, $x_j'$ are complex numbers. We now apply the projective automorphism $[x : y : z] \mapsto [e^{i\theta} x : e^{i\theta} y : z]$; the common point is still $[1 : 0 :0]$, and this map has the effect of multiplying the complex numbers $x_j$, $x_j'$ by the phase $e^{i\theta}$. For all but finitely many values of $\theta$, the transformed complex numbers $\{x_j\}$ and $\{x_j'\}$ will have pairwise distinct real parts. After relabeling the indices, we let
\begin{equation*}
    \Re(x_1) < \cdots < \Re(x_A),\quad \Re(x_1') < \cdots < \Re(x_B').
\end{equation*}
 Each point on $\ell_3$ off of the common point $[1 : 0 : 0]$ is of the form $[k : -1 : 0]$, so we let 
\begin{equation*}
    \ell_3 \cap S = \{[k_1 : -1 : 0], \ldots, [k_C : -1 : 0]\}
\end{equation*}
where $C \geq A,B$. 

Let $k \in \{k_1,\ldots,k_C\}$; consider the pencil of lines passing through $[k : -1 : 0]$ and a finite point of $S$. Each line in this pencil is of the form $x + k y = d$, with $k$ the slope and $d$ the $x$-intercept. Let $d^*$ be the $x$-intercept of this pencil with minimal real part. Because we have assumed $S$ determines no ordinary lines, the line $x + ky = d^*$ must pass through at least two finite points of $S$. So there must be some $x \in \{x_1,\ldots,x_A\}$ and $x' \in \{x_1',\ldots,x_B'\}$ such that $x + ky_1 = x' + ky_2 = d^*$. We must have $x = x_1$ and $x' = x_1'$, because otherwise, 
\begin{equation*}
    \Re(x_1 + ky_1) < \Re(x + ky_1) = \Re(d^*) \text{ or } \Re(x_1' + ky_2) < \Re(x' + ky_2) = \Re(d^*)
\end{equation*}
which violates our hypothesis on $d^*$. Here we use the fact that ordering complex numbers by their real part respects addition, and we avoid the fact that this ordering does not respect multiplication by writing the complex line $x + ky = d$ so that $x$ has coefficient one. See Figure \ref{fig:kn_fig} for a depiction of this situation.

Because the line $x + ky = d^*$ passes through $(x_1, y_1)$ and $(x_1', y_2)$, we have $x_1 + k y_1 = x_1' + k y_2$, and so $k = -\frac{x_1-x_1'}{y_1-y_2}$. Thus there is only one possible value of $k$, so the line $\ell_3$ passes through only one point of $S$ (not including the point of concurrency). We conclude that $C = 1$, and because we chose the line at infinity to have the most points, $A = B = C = 1$. Including the point of concurrency, there are at most four points in $S$, and therefore $S$ cannot be a Sylvester-Gallai configuration.
\end{proof}

It is worth noting that although we order complex numbers by their real part in the above proof, all we need is an ordering that respects addition and multiplication by positive real numbers. The orderings over $\CC$ with these properties are all of the form $x > y$ if and only if $\Re (e^{i\theta}(x-y)) > 0$, so ordering points by the real coordinate is fully general. 

\subsection{Proof of Theorem \ref{concurrent_bound}}

We now prove Theorem \ref{concurrent_bound}. We generalize the idea presented above to show that a Sylvester-Gallai configuration lying on $m$ concurrent lines contains at most $m-2$ points on each line. We build up to this bound in steps, first proving that each line contains at most $\binom{m-1}{2}$ points, then proving that each line contains at most $3m-9$ points (for $m \geq 4$), and finally that each line contains at most $m-2$ points. The initial bound is obtained by proving that the points on one line inject into the edges of a complete graph on $m-1$ vertices. The $3m-9$ bound comes from proving this graph must actually be planar, and the $m-2$ bound comes from proving this graph is acyclic, and thus has at most $m-2$ edges. The proof that the graph is acyclic involves a surprising application of Green's identity from multivariable calculus. 

\subsubsection*{Proof setup}

Let $S \subset \CC^2$ be a Sylvester-Gallai configuration lying on a family $\ell_1,\ldots,\ell_m$ of concurrent lines, with $\ell_m$ having the most points of $S$. After applying a projective automorphism we may assume $\ell_m$ is the line at infinity, and the common point is $[1 : 0 : 0]$. Then the lines $\ell_a$, $1 \leq a \leq m-1$ are of the form $y = y_a$. The finite points of $S$ are of the form $(x,y_a)$, and the infinite points (not including the common point $[1 : 0 : 0]$) are of the form $[k : -1 : 0]$. After applying an automorphism as in \S\ref{section:simplified_kn}, we may assume that the $x$-coordinates of finite points in $S$ have distinct real parts. We distinguish the points on each line $\ell_a$ having minimal real coordinate:
\begin{equation*}
    x_a^* = \argmin_{(x,y_a) \in S} \Re(x).
\end{equation*}

Our proof will focus on the points 
\begin{equation*}
    S^* =  \{(x_1^*, y_1), \ldots, (x_{m-1}^*, y_{m-1})\}
\end{equation*}
and we will mostly ignore the other finite points.

\begin{figure}
\includegraphics[width=300px]{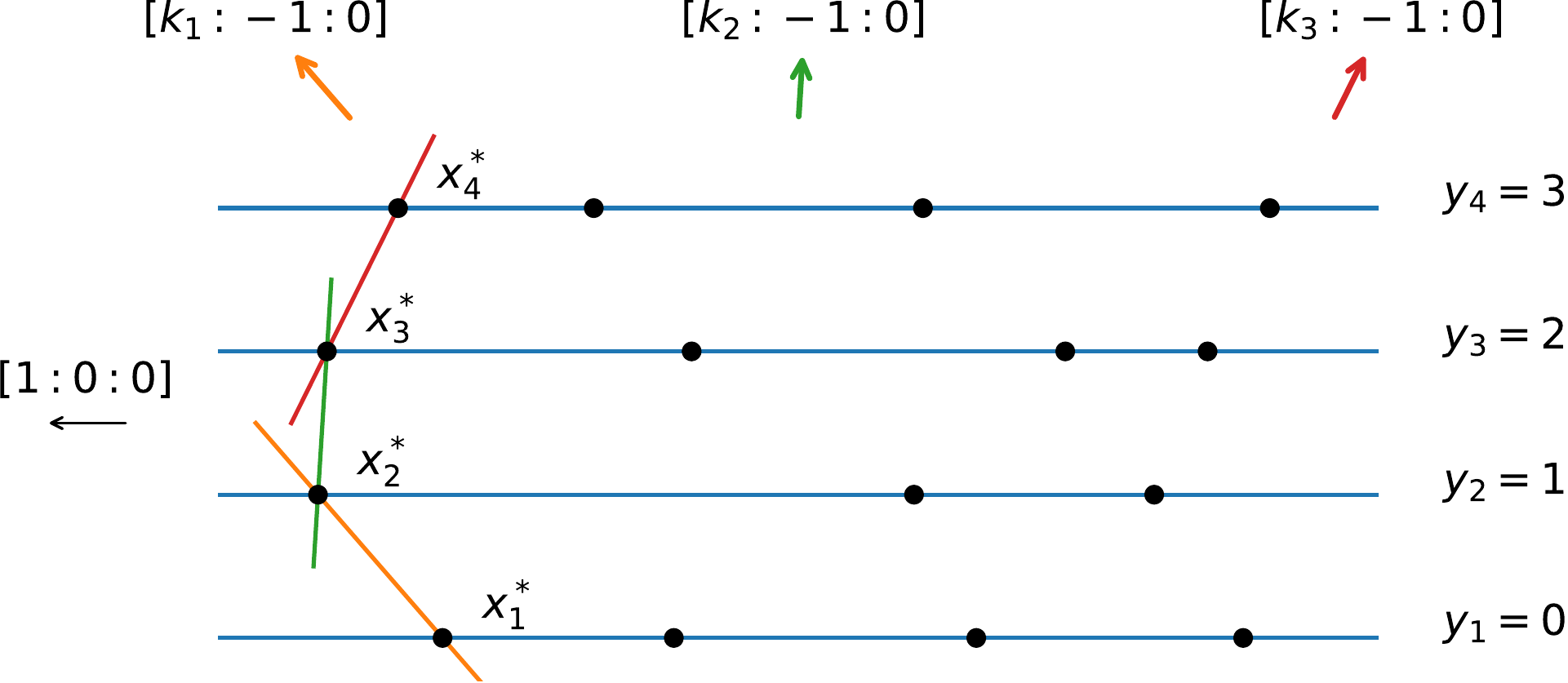}
\caption{Four concurrent lines, with the points on each line ordered by their real part. The colored lines have minimal $x$-intercept.}
\label{fig:four_concurrent}
\end{figure}

\subsubsection*{Constructing the graph on finite lines}
We now construct a graph $G$ on the $m-1$ vertices $\{y_1,\ldots,y_{m-1}\}$. We will add one edge for each point $[k : -1 : 0]\in S$ on the line at infinity. For each such point, the pencil of lines through it consists of lines of the form $x + ky = d$. We let $d_k^*$ be the $x$-intercept arising in this way with minimal real part. Then $x + ky = d_k^*$ can only pass through points of $S^*$. Indeed, suppose $(x, y_a)$ lies on the line $x + ky = d_k^*$. If $x \neq x_a^*$, then
\begin{equation*}
    \Re(x_a^* + ky_a) < \Re(x + ky_a) = \Re(d_k^*)
\end{equation*}
which violates our assumption on $d_k^*$. Now, because we assume $S$ admits no ordinary lines, the line $x + ky = d_k^*$ must pass through at least two points of $S^*$. We will choose two such points arbitrarily, say $(x_a^*, y_a)$ and $(x_b^*, y_b)$, and admit the edge $(y_a, y_b)$ to the graph $G$. Notice that 
\begin{equation*}
     k = -\frac{x_b^* - x_a^*}{y_b - y_a}
\end{equation*} 
so the point at infinity is determined by the edge---thus each point on the line at infinity will correspond to a different edge in the graph. Put differently, the edges of $G$ are in bijection with the points of $S$ lying on $\ell_m$. See Figure \ref{fig:four_concurrent} for a depiction of this situation. 

Notice that because $G$ has $m-1$ vertices, it has at most $\binom{m-1}{2}$ edges, so the line at infinity contains at most $\binom{m-1}{2}$ points. Already this is an interesting result, and is enough to prove Conjecture 4.6 from \cite{Z}. We will improve this result by investigating the graph $G$ further: first we will show $G$ is planar, and then we will show $G$ is acyclic, which implies that $G$ has at most $3m-9$ (for $m \geq 4$) and at most $m-2$ edges correspondingly.

\subsubsection*{Graph adjustments and properties} 
We apply some transformations to the set $S$ to simplify the rest of the proof. First we remove from consideration all vertices $y_a$ that are not adjacent to an edge in the graph $G$. Next, if necessary, we apply the projective automorphism from \S\ref{section:simplified_kn}, $[x : y : z] \mapsto [e^{i\theta} x : e^{i\theta} y : z]$, so that each of the values
\begin{equation*}
    \frac{1}{y_b-y_a}\det \begin{pmatrix}
        1 & 1 & 1 \\ 
        x_a^* & x_b^* & x_{c}^* \\
        y_a & y_b & y_{c}
    \end{pmatrix},\quad 1 \leq a,b,c \leq m-1
\end{equation*}
has real part 0 only if it is equal to 0. Notice that this automorphism keeps the points at infinity fixed and changes each of these determinants by a factor of $e^{i\theta}$, so the desired property holds for all but finitely many values of $\theta$. Geometrically, we are rotating the complex plane so none of these values align with the imaginary axis.
We do this for the following reason. If $x + ky = d$ is the line through $(x_a^*, y_a)$ and $(x_b^*, y_b)$, then 
\begin{equation*}
    -\frac{1}{y_b-y_a}\det \begin{pmatrix}
        1 & 1 & 1 \\ 
        x_a^* & x_b^* & x_{c}^* \\
        y_a & y_b & y_{c}
    \end{pmatrix} = x_{c}^* + ky_{c} - d. 
\end{equation*}
Thus our condition ensures that $\Re(x_{c}^* + ky_{c}) = \Re(d)$ only if $x_{c}^* + ky_{c} = d$, in which case $(x_a^*, y_a)$, $(x_b^*, y_b)$, $(x_{c}^*, y_{c})$ are collinear.

\begin{lemma}\label{lem:graph_G_props}
The graph $G$ satisfies the following.
\begin{enumerate}[label=(\alph*)]
    \item \label{lempt:G_1_eq_collinear} For $(y_a,y_b) \in G$ an edge, let $x + ky = d$ be the line passing through $(x_a^*, y_a)$ and $(x_b^*, y_b)$. Then 
    \begin{equation*}
        \Re(x_{c}^* + ky_{c}) \geq \Re(d),\quad 1\leq c \leq m-1
    \end{equation*}
    with equality only if $(x_a^*, y_a)$, $(x_b^*, y_b)$, and $(x_{c}^*, y_{c})$ are collinear. 

    \item \label{lempt:G_2_not_collinear} For $(y_a, y_b),(y_c,y_d) \in G$ two edges,  
    \begin{equation*}
       \frac{x_b^* - x_a^*}{y_{b} - y_a} \neq \frac{x_{d}^* - x_c^*}{y_{d} - y_c}
    \end{equation*}
\end{enumerate}
\end{lemma}
\begin{proof}
Property \ref{lempt:G_1_eq_collinear} follows from the definition of $G$ and the discussion prior to this lemma. Property \ref{lempt:G_2_not_collinear} is satisfied because we added precisely one edge to the graph for a given slope $k=-\frac{x_b^*-x_a^*}{y_b-y_a}$. 
\end{proof}
One should ignore the line at infinity at this point, and think of the graph as just arising from a set of points in $\CC^2$. Lemma \ref{lem:graph_G_props}\ref{lempt:G_1_eq_collinear} is the determining characteristic of the graph. 

Before proceeding, it is worth noting that in the real case, $G$ must be a disjoint union of paths. For example, in Figure \ref{fig:four_concurrent}, the graph $G$ has edges $(y_1, y_2)$, $(y_2, y_3)$, and $(y_3, y_4)$, so $G$ is a path on four vertices. The proof is not so hard: one can show that each vertex $y_a$ is adjacent to at most one value $y > y_a$ and at most one value $y < y_a$, thus each vertex has degree at most two. In addition, the graph $G$ has no cycles; indeed, the minimal $y$ value appearing in a cycle can only be connected to one other vertex, yielding a contradiction. These two facts---that vertices have degree at most two and $G$ has no cycles---proves that in the real case, $G$ is a disjoint union of paths. In the complex case we will establish that our graph $G$ is acyclic---this is more difficult. In the real case, we can analyze the structure of the graph by considering just the $y$-coordinates. No such analysis is possible in the complex case: whereas over the real numbers the vertex with minimal $y$ value must have degree one in $G$, over the complex numbers we cannot single out a vertex that must have degree one just by looking at the $y$ coordinates. Instead, we must consider the $y$ coordinates and $x$ coordinates together. Additionally, we must use all of this data to establish a global condition---acyclicity---as opposed to a local degree bound. We have sketched this comparison to the real numbers solely for the sake of illustration, we do not know of any interesting applications in the real case. 

\subsubsection*{The function $u$}
We define a piecewise linear convex function $u: \CC \to \RR$ which allows us to analyze whole regions of $\CC$ in our argument, rather than restricting our attention to the points $\{y_1,\ldots,y_{m-1}\} \subset \CC$; it will play a key role in the rest of the proof. 
\begin{definition}\label{u_defn}
We let $u: \CC \to \RR$ be given by
\begin{equation*}
    u(y) = \sup_{x+ky=d} \Re( d - ky )
\end{equation*}
where the supremum is taken over all lines $x + ky = d$ occuring from an edge in $G$; that is, pairs $(k,d)$ where
\begin{equation*}
    x_a + k y_a = x_b + ky_b = d
\end{equation*}
for $(y_a, y_b) \in G$. 
\end{definition}
Note that $u$ is convex as a function on $\RR^2$, because it is defined as the pointwise supremum of affine linear functions. Later we will apply a lemma from real analysis to this function in order to prove $G$ is acyclic. 

\begin{lemma}\label{u_props}
The function $u$ satisfies
\begin{enumerate}[label=(\alph*)]
    \item \label{lempt:u_1_yv} $u(y_a) = \Re(x_a^*)$ for all $y_a \in G$;
    \item \label{lempt:u_2_convex} $u(\lambda y + (1-\lambda) y') \leq \lambda u(y) + (1-\lambda) u(y')$, $y,y' \in \CC, 0 \leq \lambda \leq 1$;
    \item \label{lempt:u_3_edge} $u(\lambda y_a + (1-\lambda) y_b) = \lambda \Re(x_a) + (1-\lambda) \Re(x_b)$, $0 \leq \lambda \leq 1$;
    \item \label{lempt:u_4_polyreg} If $(x_a^*, y_a)$, $(x_b^*, y_b)$ lie on the line $x + ky = d$, then 
    \begin{equation*}
        u(y) = \Re (d - ky)
    \end{equation*}
    in a polygonal region whose interior contains the open line segment connecting $y_a$ and $y_b$. 
\end{enumerate}
\end{lemma}
\begin{proof}
Property \ref{lempt:u_1_yv} follows from Lemma \ref{lem:graph_G_props}\ref{lempt:G_1_eq_collinear}. Indeed,
\begin{equation*}
    \Re(x_a^* + ky_a) \geq \Re(d)
\end{equation*}
for any such $k, d$, implying $\Re(x_a^*) \geq \Re(d - ky_a)$ as desired. 
Property \ref{lempt:u_2_convex} is convexity:
\begin{align*}
    u(\lambda y + (1-\lambda) y') &= \sup_{x+ky=d} \Re( d - k(\lambda y + (1-\lambda) y') ) \\
    &= \sup_{x+ky=d} \Re( \lambda (d - ky) + (1-\lambda) (d - ky') ) \\
    & \leq \lambda \sup_{x+ky=d} \Re(d-ky) + (1-\lambda) \sup_{x+ky=d} \Re(d-ky') \\
    &= \lambda u(y) + (1-\lambda) u(y').
\end{align*}
For property \ref{lempt:u_3_edge}, 
\begin{equation*}
    u(\lambda y_a + (1-\lambda) y_b) \leq \lambda \Re(x_a) + (1-\lambda) \Re(x_b)
\end{equation*}
by convexity and property \ref{lempt:u_1_yv}, and by the definition of $u$, 
\begin{equation*}
    u(\lambda y_a + (1-\lambda) y_b) \geq \Re(d - k (\lambda y_a + (1-\lambda) y_b)) = \lambda \Re(x_a) + (1-\lambda) \Re(x_b)
\end{equation*}
yielding equality. 

For \ref{lempt:u_4_polyreg}, observe that $f(y) = \Re(d - ky)$ is one of the affine linear functions the infimum in Definition \ref{u_defn} runs over. It is clear that $u$ is piecewise linear in polygonal regions. Those polygonal regions cannot intersect the segments $y_ay_b$ transversally, as this would contradict part \ref{lempt:u_3_edge}. So if $u(y)$ is not equal to $\Re(d - ky)$ in a region properly containing the segment $y_ay_b$, there is some other line $x + k'y = d'$ passing through $(x_c^*, y_c)$, $(x_d^*, y_d)$ and with corresponding function $g(y) = \Re(d' - k'y)$ such that $g(y) = f(y)$ on the line segment connecting $y_a$, $y_b$. But then 
\begin{equation*}
    \Re(x_a^*) = \Re(d' - k'y_a),\quad \Re(x_b^*) = \Re(d' - k'y_b)
\end{equation*}
and by Lemma \ref{lem:graph_G_props}\ref{lempt:G_1_eq_collinear}, this implies the four points $(x_a^*, y_a)$, $(x_b^*, y_b)$, $(x_c^*, y_c)$, $(x_d^*, y_d)$ are collinear. Thus the edges $(y_a, y_b), (y_c, y_d) \in G$ correspond to the same line in $\CC^2$ which violates Lemma \ref{lem:graph_G_props}\ref{lempt:G_2_not_collinear}.
\end{proof}

\subsubsection*{The graph $G$ is planar}
We now move on to proving $G$ is planar; see Figure \ref{fig:planar_embedding} for an example of our planar embedding. 

% The boundary data of the function $u$ 
% \TODO boundary conditions on $u$ encode Lemma \ref{lem:graph_G_props} property (1), the key property of the graph. 

\begin{lemma}\label{lemma:G_planar}
The graph $G$ is planar. In particular, if we draw $G$ in the complex line $\CC$ by placing vertex $y_a$ at its value and drawing edges as straight line segments, we obtain a planar embedding. 
\end{lemma}
\begin{proof}
We show that if two edges cross, then the corresponding lines for those edges must have the same slope, which contradicts Lemma \ref{lem:graph_G_props}\ref{lempt:G_2_not_collinear}.
Two edges $\{y_a,y_b\}$ and $\{y_c,y_d\}$ cross if $y_a,y_b,y_c,y_d$ are all distinct and we have
\begin{equation*}
    \lambda y_a + (1-\lambda) y_b = \gamma y_c + (1-\gamma) y_d\quad \text{ with } 0 \leq \lambda \leq 1,\ 0 \leq \gamma \leq 1.
\end{equation*}
We denote this intersection point by $z \in \CC$. Let
\begin{equation*}
    x_a + k_1y_a = x_b+k_1y_b = d_1,\quad x_c + k_2y_c = x_d+k_2y_d = d_2.
\end{equation*}
By Lemma \ref{lem:graph_G_props}\ref{lempt:G_2_not_collinear}, $k_1 \neq k_2$. Consider the two affine linear functions 
\begin{equation*}
    f_1(y) = \Re(d_1 - k_1 y),\quad f_2(y) = \Re(d_2 - k_2 y).
\end{equation*}
Then by Lemma \ref{u_props}\ref{lempt:u_3_edge}, $f_1(z) = f_2(z) = u(z)$. Now consider $s_1, s_2: [0,1] \to \RR$ by
\begin{align*}
    s_1(\lambda) &= u(y) - f_1(y),\quad y = \lambda y_c + (1-\lambda) y_d\\
    s_2(\lambda) &= u(y) - f_2(y),\quad y = \lambda y_a + (1-\lambda) y_b.
\end{align*}
Then by the definition of $u$, $s_1, s_2 \geq 0$. But $s_1$ and $s_2$ are affine linear by Lemma \ref{u_props}\ref{lempt:u_2_convex}, and take a value of zero for some $\lambda \in [0, 1]$. In fact, because the values $y_a,y_b,y_c,y_d$ are all distinct, at least one of $s_1$ and $s_2$ must take a value of zero for $\lambda \in (0,1)$. If this is the case, $s_a$ must be uniformly zero, because it is greater than or equal to zero and affine linear. Suppose without loss of generality $s_1 = 0$. Then $f_1(y_c) = u(y_c)$ and $f_1(y_d) = u(y_d)$. Then 
\begin{equation*}
    \Re(x_c^* + k_1 y_c) = \Re(d_1) \text{ and } \Re(x_d^* + k_1 y_d) = \Re(d_1)
\end{equation*}
so by Lemma \ref{lem:graph_G_props}\ref{lempt:G_1_eq_collinear}, the four points $(x_a^*, y_a), (x_b^*, y_b), (x_c^*, y_c), (x_d^*, y_d)$ are all collinear. Thus $k_1 = k_2$, which contradicts Lemma \ref{lem:graph_G_props}\ref{lempt:G_2_not_collinear}. It follows that our embedding is indeed planar. 
\end{proof}
\begin{figure}\label{fig:graph_example}
\includegraphics[width=200px]{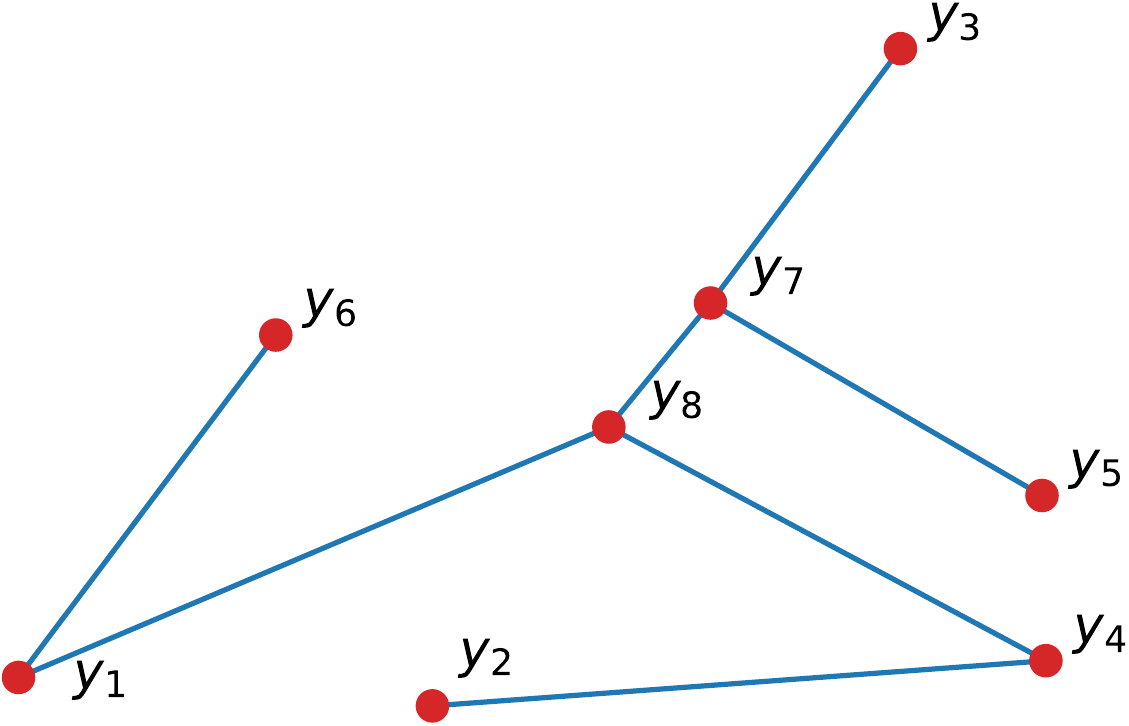}
\caption{Example of an actual graph $G$ arising from a choice of points $(x_a,y_a)$, drawn in $\CC$ according to the the value of $y$.}
\label{fig:planar_embedding}
\end{figure}
Because $G$ is planar and has at most $m-1$ vertices, an application of Euler's formula for planar graphs proves it has at most $3(m-1) - 6 = 3m-9$ edges (this holds for $m \geq 4$; for $m \leq 3$, it is clear $G$ has at most $m-2$ edges). This improves our bound on the maximum number of points on each line of the original configuration. Next we prove $G$ is acyclic, a stronger condition than being planar. 

\subsubsection*{The graph $G$ is acyclic}
Before proving $G$ is acyclic, we state a useful lemma from analysis. This lemma applies to a piecewise linear, locally convex function $v: \Omega \to \RR$, with $\Omega$ a polygonal region. By piecewise linear we mean that $\Omega$ can be partitioned into polygonal subsets such that $v$ is linear in each of these subsets. We will denote normal derivatives of $v$ at the boundary of $\Omega$ by $\frac{\partial v}{\partial \eta}$, and we consider the normal derivatives as pointing inwards toward $\Omega$. Throughout the discussion of this lemma, normal derivatives are taken in one direction only---this allows us to consider normal derivatives in different directions at points where $u$ is not differentiable. 
\begin{lemma}\label{green_formula_lemma}
For $v: \Omega \to \RR$ a piecewise linear locally convex function,
\begin{equation}\label{eq:green_formula_lem_eq}
    \int_{\partial \Omega} \frac{\partial v}{\partial \eta} ds \leq 0
\end{equation}
with equality if and only if $v$ is linear in $\Omega$. 
\end{lemma}
This lemma comes from treating $v$ as a subharmonic function. A subharmonic function has the property that the average value over each ball $B_r(x) = \{y\ |\ |y - x| \leq r\}$ is greater than or equal to the value at the center of that ball. We say $v$ is strictly subharmonic if that inequality is strict for one of these balls. If $v$ is a $C^2$ subharmonic function, then $\Delta v \geq 0$ (recall that $\Delta = \nabla \cdot \nabla$ is the Laplace operator), and if $v$ is strictly subharmonic, $\Delta v > 0$ somewhere. By Green's identity from multivariable calculus, for $v \in C^2(\Omega) \cap C(\overline{\Omega})$ strictly subharmonic,
\begin{equation*}
    \int_{\partial \Omega} \frac{\partial v}{\partial \eta} ds = -\int_{\Omega} \Delta v dx < 0.
\end{equation*}
Intuitively, subharmonic functions must on average decrease from the boundary, so the normal derivatives must on average be negative. Convex functions are always subharmonic, and the only harmonic convex functions are the linear functions. For simplicity we prove Lemma \ref{green_formula_lemma} in the special case that $v$ is piecewise linear and convex, but the same idea applies to any sufficiently regular subharmonic function. See \cite{MT} for more information on Green's identity and subharmonic functions. 
\begin{proof}[Proof of Lemma \ref{green_formula_lemma}]
\begin{figure}
\includegraphics[width=200px]{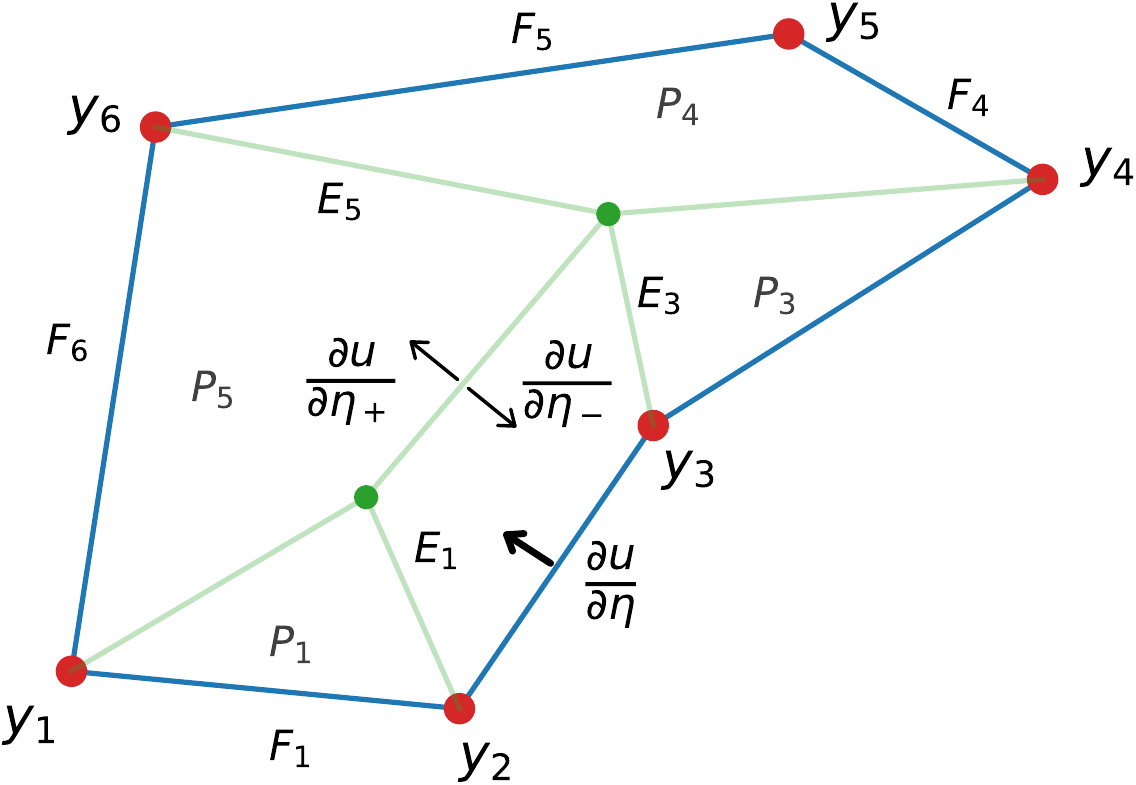}
\caption{We decompose $\Omega$ into piecewise linear regions, and apply Green's identity in each of these regions.}
\label{green_lemma_im}
\end{figure}
Let $\Omega$ be partitioned into polygonal regions $P_1, P_2, \ldots, P_n$, on which $v$ is linear. Let the bounding edges of all these polygons lying inside $\Omega$ be $E_1, E_2, \ldots, E_k$, and let the bounding edges of $\Omega$ be $F_1, F_2, \ldots, F_r$. See Figure \ref{green_lemma_im} depicting this decomposition. Now that we have partitioned $\Omega$ in this way, $v$ is of class $C^2$  in each of our regions, so we are in position to apply Green's identity:
\begin{equation*}
    \int_{\partial P_j} \frac{\partial v}{\partial \eta}\, ds = \int_{P_j} \Delta v\, dx = 0.
\end{equation*}
Summing this identity over all polygons, we obtain an integral over all edges $E_j$, $F_j$.
\begin{align*}
    0 &= \sum_j \int_{\partial P_j} \frac{\partial v}{\partial \eta} ds \\
    &= \sum_j \int_{F_j} \frac{\partial v}{\partial \eta} ds + \sum_j \int_{E_j} \left(\frac{\partial v}{\partial \eta_+} + \frac{\partial v}{\partial \eta_-}\right)ds \\
    &= \int_{\partial \Omega} \frac{\partial v}{\partial \eta} ds + \sum_j \int_{E_j} \left(\frac{\partial v}{\partial \eta_+} + \frac{\partial v}{\partial \eta_-}\right)ds 
\end{align*}
The values $\frac{\partial u}{\partial \eta_+}$ and $\frac{\partial u}{\partial \eta_-}$ are the normal derivatives pointing in each possible direction along the edge $E_j$; these numbers are in general different. We get the term $\frac{\partial v}{\partial \eta_+} + \frac{\partial v}{\partial \eta_-}$ because for each internal edge $E_j$, we integrate the normal derivative pointing in both possible directions. 

We claim 
\begin{equation}\label{eq:partial_both_sum}
    \frac{\partial v}{\partial \eta_+} + \frac{\partial v}{\partial \eta_-} \geq 0
\end{equation}
for each internal edge $E_j$, and if equality is achieved everywhere then $u$ is linear. For $y \in E_j$, let $f(s) = v(y + s\mathbf{\eta_+})$ where $\mathbf{\eta_+}$ is a unit vector perpendicular to $E_j$. Then $f(s)$ is a piecewise linear convex function, and
\begin{equation*}
    \frac{\partial v}{\partial \eta_+} + \frac{\partial v}{\partial \eta_-} = \frac{\partial f(s)}{\partial s_+} \Big|_{s = 0} + \frac{\partial f(s)}{\partial s_-} \Big|_{s = 0}
\end{equation*}
where $\frac{\partial}{\partial s_+}$ is the derivative pointing in the positive direction, and $\frac{\partial}{\partial s_-}$ is the derivative pointing in the negative direction. We have
\begin{equation*}
    \frac{\partial}{\partial s_+} f(s)\Big|_{s = 0} + \frac{\partial}{\partial s_-} f(s)\Big|_{s = 0} = \lim_{h\to 0} \frac{s(h) + s(-h) - 2s(0)}{h} \geq 0
\end{equation*}
because $s(h) + s(-h) - 2s(0) \geq 0$ by convexity. We have proved (\ref{eq:partial_both_sum}), and if we have equality over all edges $E_j$, then $v$ is in fact linear on all of $\Omega$. Thus 
\begin{equation*}
    \int_{\partial \Omega} \frac{\partial v}{\partial \eta} ds = -\sum_j \int_{E_j} \left(\frac{\partial v}{\partial \eta_+} + \frac{\partial v}{\partial \eta_-}\right)ds \leq 0
\end{equation*}
and equality is achieved if and only if $u$ is linear. 
\end{proof}

\begin{remark}
We will soon apply Lemma \ref{green_formula_lemma} to our helper function $u$. In this context, one interpretation of Equation (\ref{eq:green_formula_lem_eq}) is that it synthesizes all the data contained in Lemma \ref{lem:graph_G_props}\ref{lempt:G_1_eq_collinear}. 
Along the edge connecting $y_a$ to $y_b$, the tangent plane to $u$ is given by the function $f(y) = \Re(x + ky - d)$, where $x + ky = d$ passes through $(x_a,y_a)$ and $(x_b,y_b)$. Lemma \ref{lem:graph_G_props}\ref{lempt:G_1_eq_collinear} says that the graph of these tangent planes lie below each point $(y_a, u(y_a))$. But because $u$ is convex, any tangent plane to $u$ lies below the entire graph of $u$. In this way, the boundary data to $u$ contains the information of Lemma \ref{lem:graph_G_props}\ref{lempt:G_1_eq_collinear}, and integrating that boundary data along the entire region collects all that data into one equation. 
\end{remark}
We now prove the main lemma. 
\begin{lemma}\label{acyclic_lemma}
The graph $G$ is acyclic. 
\end{lemma}

\begin{proof}
Suppose $G$ has a cycle. After relabeling, this is a sequence of vertices $y_1,y_2,\ldots, y_n$, $n \geq 3$, where $(y_1, y_2), \ldots, (y_n, y_1) \in G$. In what follows, we will let $y_{n+1} = y_1$, $x_{n+1} = x_1$. Because the embedding described in Lemma \ref{lemma:G_planar} is planar, the points $y_1,\ldots, y_n$ and the line segments connecting them cut out a polygonal region in the plane---we order the points so that the inside of this polygon lies on the left hand side of each segment $y_{a+1} - y_a$. Let $\Omega$ be this polygonal region. Then $u$ is a locally convex function on $\Omega$, and by Lemma \ref{lem:graph_G_props}\ref{lempt:G_1_eq_collinear}, $u$ is only linear on $\Omega$ if all the points $(x_a^*, y_a)$ appearing in the cycle are collinear. This is impossible, as each line through points of $S^*$ corresponds to just one edge in $G$. So $u$ is a piecewise linear convex function on $\Omega$, which is not linear on all of $\Omega$. 
\begin{figure}
\includegraphics[width=200px]{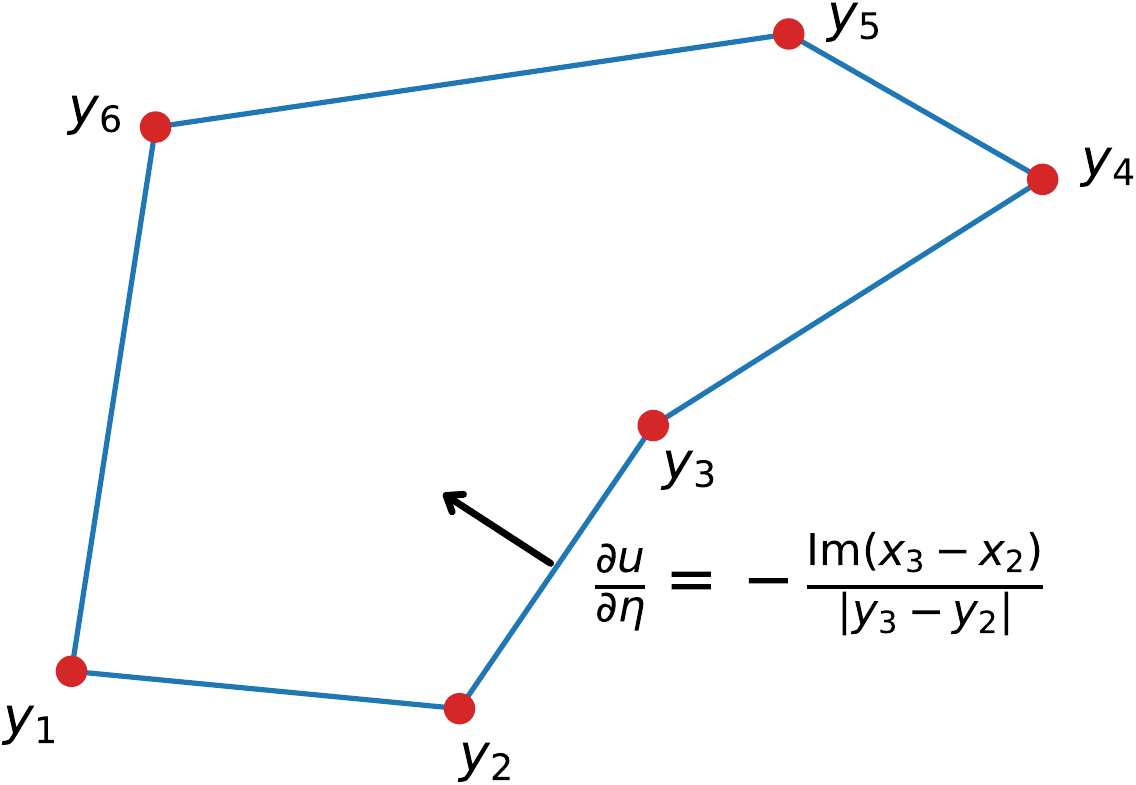}
\caption{We can compute the normal derivative of $u$ in terms of the values $x_a$.}
\label{fig:int_u_func}
\end{figure}
We may apply Lemma \ref{green_formula_lemma} to the function $u$ to obtain
\begin{equation}\label{int_u_pos_fact}
    \int_{\partial \Omega} \frac{\partial u}{\partial \eta} ds < 0.
\end{equation}
We will prove that this integral must actually equal zero, yielding a contradiction.

Lemma \ref{u_props}(c,d) gives a formula for $u$ in an open neighborhood of the segment connecting $y_a$ and $y_{a+1}$:
\begin{equation}\label{u_val_in_reg}
    u(y) = \Re\left(d - \frac{x_{a+1}^* - x_a^*}{y_{a+1} - y_a} y\right).
\end{equation}
But then the normal derivative of $u$ at a point on this segment is equal to the normal derivative of this affine linear function. The inward pointing normal vector along the edge connecting $y_a$, $y_{a+1}$ is given by 
\begin{equation*}
    v = i\frac{y_{a+1} - y_a}{|y_{a+1} -y_a|}
\end{equation*}
because multiplication by $i$ corresponds to a rotation by $\pi/2$. We now compute the normal derivative using Equation (\ref{u_val_in_reg}):
\begin{align*}
    \frac{\partial u}{\partial \eta} &= \lim_{h\to 0}\frac{u(y + hv) - u(y)}{h} \\
    &= \Re\left(-\frac{x_{a+1}^* - x_a^*}{y_{a+1} - y_a} \cdot i\frac{y_{a+1}-y_a}{|y_{a+1} - y_a|}\right) =  -\frac{\Re i(x_{a+1}^* - x_a^*)}{|y_{a+1} - y_a|} = \frac{\Im(x_{a+1}^* - x_a^*)}{|y_{a+1} - y_a|}.
\end{align*}
See Figure \ref{fig:int_u_func} for a depiction of this integrand. Substituting this equation into the integral (\ref{int_u_pos_fact}), 
\begin{equation*}
    \int_{y_a}^{y_{a+1}} \frac{\partial u (y)}{\partial \eta} ds = \int_{y_a}^{y_{a+1}} \frac{\Im(x_{a+1}^* - x_a^*)}{|y_{a+1} - y_a|} ds = \Im(x_{a+1}^* - x_a^*). 
\end{equation*}

Summing these integrals over all the edges $y_a$, $y_{a+1}$, we find
\begin{equation*}
    \int_{\partial \Omega} \frac{\partial u(y)}{\partial \eta} ds = \sum_{a=1}^n \Im(x_{a+1}^* - x_a^*) = 0.
\end{equation*}
This fact contradicts Equation (\ref{int_u_pos_fact}), and we are done.
\end{proof}

Because the graph $G$ is acyclic, and because $G$ has at most $m-1$ vertices, $G$ has at most $m-2$ edges. The edges of $G$ are in bijection with the points of $S \cap \ell_m$, and we chose $\ell_m$ to have the most points of $S$. Thus we conclude that each line $\ell_1,\ldots,\ell_m$ contains at most $m-2$ points (other than the concurrency point), establishing Theorem \ref{concurrent_bound}.

\begin{remark}\label{rmk:analysis_important}
Our use of analysis in this proof allows us to find a global obstruction to the existence of a cycle in $G$---this is important, because as we noted after Lemma \ref{lem:graph_G_props}, there is no local obstruction to the existence of a cycle. 
\end{remark}

\begin{remark}\label{rmk:int_normal}
Notice that along the edge connecting $y_a$ to $y_b$, the integral of the tangent derivative of $u$ is just the difference $u(y_b) - u(y_a) = \Re(x_b) - \Re(x_a)$. Analogously, the integral of the normal derivative of $u$ is the difference $\Im(x_b) - \Im(x_a)$. One can think of the imaginary parts of the complex numbers $x_a$ as determining these slopes: that is precisely their relevance to the function $u$. 
\end{remark}

\begin{remark}\label{acyclic_only_condition}
We believe that acyclicity is the \textit{only} condition on graphs $G$ arising from our construction. Computational evidence suggests that up to eight points, all acyclic graphs appear in this way.
\end{remark}

\section{Conclusion}
This paper establishes a sharp condition for Sylvester-Gallai configurations lying on a family of concurrent lines. This is an extremely special situation: most collections of points do not lie on a few concurrent lines. However, this is one of the few results on ordinary lines in the complex plane, and it involves a new approach to studying these limes---that of ordering complex numbers by their real part. Our hope is that this theorem will open up further study of complex Sylvester-Gallai configurations. The eventual goal of this study would be a proof of the following conjecture. 
\begin{conjecture}\label{full_structural_conjecture}
The only Sylvester-Gallai configurations in $\CC^2$ are the Fermat configurations and a finite number of exceptional examples.
\end{conjecture}
This conjecture is a slight relaxation of Problem 1.10 in Bokowski \& Pokora \cite{BP} (we allow for finitely many exceptional examples rather than allowing for only the known exceptional examples), but is a much older folklore conjecture. The Fermat configurations are extremely special: they lie on three non concurrent lines, and all other lines pass through exactly three points of the set. One could imagine many other theorems working toward Conjecture \ref{full_structural_conjecture} without proving the full result---here are some possibilities. 

\begin{conjecture}\label{op1}
If a configuration lies on $m > 3$ lines, and if each of those lines have more than $C(m)$ points, then the set admits an ordinary line.
\end{conjecture}

This conjecture is an analogue of Theorem \ref{concurrent_bound} for arbitrary lines, but seems substantially more difficult. In particular, one has to use the fact that all the lines have many points, as the Fermat configurations are an infinite family of Sylvester-Gallai configurations lying on only three lines. It is plausible that our approach to Theorem \ref{concurrent_bound} could be pushed further to prove results on non concurrent lines. A difficulty, however, is that the concurrency assumption plays a crucial role early in our proof: this assumption is what allowed us to ignore most of the finite points, and focus on a set of $m-1$ finite points, one selected from each line. 

\begin{conjecture}\label{op2}
Aside from the Hesse configuration, there are no complex Sylvester-Gallai configurations where each line passes through exactly three points.
\end{conjecture}

Combinatorially, a configurations of points and lines where every line passes through three points is called a \textit{Steiner triple system}. This conjecture states that the only Steiner triple system that can be embedded in $\CC^2$ is the Hesse configuration. Limbos \cite{Limbos} established this conjecture up to 15 points, and it follows from Hirzebruch's inequality that for any $m \geq 4$, there is no complex configuration with exactly $m$ points on every line. 
Conjecture \ref{op2} runs in the opposite direction to Conjecture \ref{op1}. Rather than showing that Sylvester-Gallai configurations cannot have too many points on several lines, the goal here is to prove that some line must contain many points. Indeed, if Conjecture \ref{full_structural_conjecture} is true, another statement similar in spirit to Conjecture \ref{op2} must be true as well: because the Fermat configurations have three lines passing through $n / 3$ points, any large enough Sylvester-Gallai configuration must have many points lying on one line. Beyond the open problems stated here, there are many avenues for future research, and we hope to see further results on complex Sylvester-Gallai configurations.

\section{Acknowledgements}
Many thanks to Frank de Zeeuw for suggesting the problem and for helpful discussions along the way. Thanks to the Baruch Combinatorics REU and organizer Adam Scheffer for supporting this work and providing research mentorship. Thanks to Wilhelm Schlag for providing suggestions with regard to Lemma \ref{green_formula_lemma}. The author would also like to thank an anonymous reviewer for extremely helpful and detailed comments.

\end{document}